\documentclass[11pt,leqno]{amsart}
\usepackage{amscd}
\usepackage{color}
\usepackage[symbol]{footmisc}
\usepackage{amssymb}
\usepackage{amsfonts}
\usepackage{latexsym}
\usepackage{verbatim}
\usepackage{bbm}

\renewcommand{\epsilon}{\varepsilon}
\renewcommand{\theta}{\vartheta}
\renewcommand{\phi}{\varphi}

\newcommand{\vp}{\varphi}

\newcommand{\ddbar}{\sqrt{-1} \partial \overline{\partial}}

\newcommand{\MA}{Monge-Amp\`{e}re}

\newcommand{\ri}{\rightarrow}
\newcommand{\DP}{Dirichlet problem}

\newcommand{\arccot}{\mathrm{arccot}}

\newcommand{\Hess}{\mathrm{Hess}_{\mathbb{C}} }

\begin{document}
	
		\newcounter{theor}
		\setcounter{theor}{1}
		\newtheorem{claim}{Claim}[section]
		\newtheorem{theorem}{Theorem}[section]
		\newtheorem{lemma}[theorem]{Lemma}
		\newtheorem{corollary}[theorem]{Corollary}
		\newtheorem{proposition}[theorem]{Proposition}
		\newtheorem{prop}[theorem]{Proposition}
		\newtheorem{question}{question}[section]
		\newtheorem{defn}[theorem]{Definition}
		\newtheorem{remark}[theorem]{Remark}
		\numberwithin{equation}{section}
		\title[The Cauchy-Dirichlet problem for parabolic dHYM equation]{The Cauchy-Dirichlet problem for parabolic deformed Hermitian-Yang-Mills equation}

		\author[L. Huang]{Liding Huang}
		\address{Westlake Institute for Advanced Study (Westlake University), 18 Shilongshan Road, Cloud Town, Xihu District, Hangzhou, P.R. China}
		\email{huangliding@westlake.edu.cn}

		\author[J. Zhang]{Jiaogen Zhang}
		\address{Jiaogen Zhang, School of Mathematical Sciences, University of Science and Technology of China, Hefei 230026, People's Republic of China }
		\email{zjgmath@ustc.edu.cn}
		\subjclass[2010]{58J35, 35J60, 53C07, 35B45}
		\keywords{deformed Hermitian-Yang-Mills equation, energy functional, heat flow}
		
		\begin{abstract}
			The purpose of this paper is to investigate  the parabolic deformed Hermitian-Yang-Mills equation with hypercritical phase in a smooth domain $\Omega\subset \mathbb{C}^{n}$. By using  $J$-functional, we are able to prove the convergence of solutions. As an application, we give an alternative proof of the Dirichlet problem for deformed Hermitian-Yang-Mills equation.
		\end{abstract}
		\maketitle

		\section{Introduction}
		Let $\Omega\subset \mathbb{C}^{n}$ be a smooth domain with smooth boundary $\partial \Omega$.  For any $u\in C^{2}(\overline{\Omega})$, we let $\Hess\,u=\ddbar u$ be the complex Hessian of $u$ and $\lambda=\lambda(\Hess\, u)$ be the eigenvalues of $\Hess\, u$. Throughout the paper, for each $\lambda=(\lambda_{1},\cdots,\lambda_{n})$ and $\lambda(\Hess\, u)=(\lambda_{1}(\Hess\, u),\cdots,\lambda_{n}(\Hess\, u))$, we denote
		\begin{equation*}
			\Theta(\lambda)=\sum_{i=1}^{n}\arccot\,\lambda_{i},\qquad \Theta(\Hess\, u)=\Theta(\lambda(\Hess\, u))
		\end{equation*}
		when no confusion will occur.

		The Dirichlet problem for the deformed Hermitian-Yang-Mills (dHYM in short) equation  can be written in the following form
		\begin{equation}\label{dp}
			\left\{\begin{array}{ll}
				\Theta(\Hess\, u)=\hat{\theta}\quad &\mathrm{in}~{\Omega}, \\[1mm]
				u=\varphi \quad &\mathrm{on}~ \partial{\Omega}.\\[1mm]
			\end{array}\right.
		\end{equation}
		Here $\varphi$ is a  smooth functions on $\partial\Omega$. First, we may extend $\varphi$ to a smooth function (still denoted by $\varphi$ for convenience) such that
		\begin{equation}\label{initial}
			0<\Theta(\Hess\, \varphi)<\frac{\pi}{2}, \quad \mathrm{on\,} \ \overline{\Omega}.
		\end{equation}
		We say the  \eqref{dp} is  \textit{supercritical} if $\hat{\theta}\in (0,\pi)$, and  of \textit{hypercritical} if $\hat{\theta}\in (0,\frac{\pi}{2})$. In the sequel we are interested in the later case.

		The Dirichlet problem \eqref{dp} was first posed by Harvey-Lawson \cite{HL82}. Under a certain assumption on the geometry of $\partial \Omega$, Caffarelli-Nirenberg-Spruck \cite{CNS85} established the existence of solutions. Viscosity solutions were also considered by Harvey-Lawson \cite{HL09}.  Collins-Picard-Wu \cite{CPW17} proved the existence of classical solutions provided there exists a subsolution on $\Omega$.  For more related works, we refer to \cite{CKNS85,Gu98, Guan14, Tru95} and references therein.

		The equation  \eqref{dp} is a local version of the dHYM equation over closed manifolds, which has  also been extensively researched motivated by  mirror symmetry.
		Jacob-Yau \cite{JY} initially studied the dHYM equation and it was solved by Collins-Jacob-Yau \cite{CJY} under the existence of a notion of subsolution in the case of supercritical phase. Recently, Collins-Yau \cite{CY}  introduced an infinite dimensional GIT approach to the dHYM equation. One of  the important roles in this theory is the so-called geodesic equation which can be reduced to a Dirichlet problem for the degenerate dHYM equation.  
		

		On the other hand, to study the dHYM equation, the theory of geometric flows have been well developed at the same time.   A line bundle version of the Lagrangian mean curvature flow was introduced by Jacob-Yau \cite{CY}.  Takahashi \cite{TR} defined the tangent Lagrangian phase flow, which is expected to be effective  in the GIT framework. Recently Fu-Yau-Zhang \cite{FYZ20} considered a new geometric flow which resolved the dHYM equation on compact K\"ahler manifolds. The twisted version of Fu-Yau-Zhang's flow  can be also used to prove the equivalence between coerciveness  of the $J$-functional and the existence of solutions to the dHYM equations (see Chu-Lee \cite{CL}).

		Before stating our main results it is necessary for us to collect some notations. For each $T>0$,
		\[\mathcal{Q}_{T}=\Omega\times (0,T],\, \, \mathcal{B}=\Omega\times \{0\},\, \, \mathcal{S}=\partial\Omega\times[0,T],\, \, \partial_{p}\mathcal{Q}_{T}=\mathcal{B}\cup  \mathcal{S}.  \]
		Here, $\mathcal{B}$ is the $bottom$ of $\mathcal{Q}_{T}$, $\mathcal{S}$ is the $side$ of $\mathcal{Q}_{T}$ and $\partial_{p}\mathcal{Q}_{T}$ is the $parabolic$ $boundary$ of $\mathcal{Q}_{T}$.

	Inspired by the work in \cite{FYZ20},	in this paper we shall study the Cauchy-Dirichlet problem about parabolic dHYM equation 
		\begin{equation}\label{flow}
			\left\{\begin{array}{ll}
				\partial_{t}u=\cot	\Theta(\Hess\, u)-\cot\hat{\theta}\quad &\mathrm{in}~{\mathcal{Q}_{T}}, \\[1mm]
				u=\psi\quad &\mathrm{on}~ \partial_{p}\mathcal{Q}_{T},\\[1mm]
			\end{array}\right.
		\end{equation}
		where $\partial_{t}u=\frac{\partial u}{\partial t}$, $\Hess\,u$ denotes the spatial Hessian of $u$, $\hat{\theta}$ is a given constant, and  $\psi=\psi(x,t)$ is a smooth function on $\partial_{p}{\mathcal{Q}_{T}}$ satisfies 
		\begin{enumerate}
			\item[C1)]  $\psi(\cdot,0)=\varphi$ on $\overline{\Omega}$, \vspace{1mm}
			\item[C2)] $0<\theta_0<\Theta(\Hess\, \psi)<\frac{\pi}{2}-\theta_0$ on $\mathcal{S}$ with  $\theta_{0}\in(0,\,\frac{\pi}{4})$ is a given constant for any $T>0$. \vspace{1mm}
		\end{enumerate}
		We remark that such function $\psi$ is always exists, for instance, we can choose $\psi(\cdot,t)=\varphi(\cdot)$ be a constant valued function in $[0,T]$.

		The main result concerns the problems of the long time existence of flow \eqref{flow} and the solvability of \eqref{dp}. The statement is as follows.
		\begin{theorem}\label{main theorem}
			Suppose there exists a parabolic subsolution $\underline{u}\in C^{2,1}(\overline{\Omega}\times [0,\,\infty))$ for the flow \eqref{flow}, i.e. for every $T>0$,
			\begin{equation}\label{subsolution-1}
				\left\{\begin{array}{ll}
					\partial_{t}\underline{u}\leq 	\cot\Theta(\mathrm{Hess}\, \underline{u})-\cot \hat{\theta} &\mathrm{in}~{\mathcal{Q}_{T}}, \\[1mm]
					\underline{u}\leq \varphi \quad \mathrm{on}~\mathcal{B} \quad\textrm{and}\quad \underline{u}= \psi  &\mathrm{on}~\mathcal{S}.\\[1mm]
				\end{array}\right.
			\end{equation}
			Assume that $\Theta(\Hess \,\underline{u})\in(\theta_0, \frac{\pi}{2}-\theta_0)$.
			Then the solution $u$ of flow \eqref{flow} exists for all time. Furthermore, as $t\rightarrow +\infty$, $u(\cdot,t)$ approaches to the unique solution $u_{\infty}$ for \eqref{dp} in $C^{\infty}(\overline{\Omega})$.
		\end{theorem}
	
	\begin{remark}
	    Here we consider the hypercritical case since we need this assumption in the proof of $C^0$ estimates (see Proposition \ref{C0} ).
	\end{remark}
		The paper is organized as follows. In section 2, we first recall the definition and some  properties of parabolic subsolutions which play important roles in the proof of global $C^{2}$ estimates. Secondly, using the maximum principle, we prove the $C^{0}$ and the $C^{1}$ estimates. Finally, the main challenge is to prove the boundary second order estimates. It can be solved by employing the barrier function in \cite{Guan14} and the idea of Trudinger \cite{Tru95}. In section 3 we generalize the definitions of  Calabi-Yau functional and $J$-functional over smooth domains in $\mathbb{C}^{n}$. Using the monotonicity of $J$-functional along \eqref{flow},   we obtain the convergence of flow \eqref{flow} via the arguments of \cite{HouLi,PS}.
		
\bigskip

\noindent {\bf Acknowledgements.} The authors would like to thank the referees for
many useful suggestions and comments. 
		
		\section{A priori estimates}
		By a simple computation, we get
		\begin{equation*}
			\cot\Theta(\Hess\,u)=\frac{\textrm{Re}\det(\sqrt{-1}\delta_{kl}+u_{k\bar{l}})}{\textrm{Im}\det(\sqrt{-1}\delta_{kl}+u_{k\bar{l}})}.
		\end{equation*}
		Let us denote
		\begin{equation*}
			F^{i\bar{j}}(\Hess\,u)=\frac{\partial}{\partial u_{i\bar{j}}}\cot\Theta(\Hess\,u),
		\end{equation*}
		which determines a positive definite Hermitian form. Then
  \[Lu=\partial_{t}u-F^{i\bar{j}}u_{i\bar{j}}\] is  a parabolic linearized operator of \eqref{flow}.
		We assume that there exist solutions of the flow \eqref{flow} on $\Omega\times [0,\, T')$.  
		In this section, we will prove the estimates of solutions on $\mathcal{Q}_T$, where $T<T'$,
		\subsection{Subsolutions}
		Let us
		denote $\Gamma_{n}$ be the positive orthant in $\mathbb{R}^{n}$ and
		\begin{equation*}
			\Gamma=\{\lambda\in \mathbb{R}^{n}:\  0<\Theta(\lambda)< \frac{\pi}{2}\},\qquad
			\Gamma^{\sigma}=\{\lambda\in \Gamma: 0<\Theta(\lambda)< \sigma\},
		\end{equation*}
		where $\sigma\in(0, \frac{\pi}{2})$ is a constant.
		\begin{defn}\cite{SG}\label{sub}
			We say that a smooth function $\underline{u}:\overline{\Omega}\rightarrow \mathbb{R}$ is a $\mathcal{C}$-subsolution of \eqref{dp} if at each point $x\in \overline{\Omega}$, the set
			\begin{equation*}
				\left\{\lambda\in \Gamma: \Theta(\lambda)=\hat{\theta}, \text{\ and\ }  \lambda-\lambda(\Hess\,\underline{u})\in \Gamma_{n}\right\}
			\end{equation*}
			is bounded.
		\end{defn}
		For the   \eqref{dp}, Collins-Jacob-Yau \cite{CJY}  (see also  Collins-Picard-Wu \cite{CPW17}) provided an explicit description of subsolutions.
		\begin{lemma}\label{subsolution}
			A smooth function $\underline{u}: \overline{\Omega}\rightarrow \mathbb{R} $ is a subsolution of  \eqref{dp} if and only if at each point $x\in \overline{\Omega}$,  
			\begin{equation*}
				\max_{1\leq j\leq n}	\sum_{i\neq j }\arccot\, \lambda_{i}(\Hess\,\underline{u})(x)<\hat{\theta}.
			\end{equation*}
		\end{lemma}

		We shall recall the terminology of parabolic subsolutions for the parabolic dHYM equations, as already defined in \cite[Definition 1]{PT17}. 
		\begin{defn}\label{parabolic sub}  
			We say a function $\underline{u}_{p}\in C^{2,1}(\mathcal{Q}_{T})$ is a parabolic $\mathcal{C}$-subsolution for \eqref{flow}, if   there exist uniform constants $\delta_{0}, R_0>0$, such that for any $(x, t)\in\mathcal{Q}_T$,
			\begin{equation}\label{cone condition}
				\cot		\Theta(\lambda(\Hess\,\underline{u}_{p})+\mu )- \partial_{t}\underline{u}_{p}+\tau=\cot\hat{\theta}, ~ \mu+\delta_{0} I\in \Gamma_{n}~\textrm{and}~\tau>-\delta_{0}
			\end{equation}
			implies that $|\mu|+|\tau|<R_0 $. Here  $I$ denotes the vector $(1,\cdots,1)$ of $n$-tuples.
		\end{defn}


		
		The following proposition is principle, which is inspired by the recent work of Phong-T\^o \cite{PT17}.
		\begin{proposition}\label{key}
			Let $\delta_{0}, R_0>0$ be uniform constants as in Definition \ref{parabolic sub}  and $\underline{u}$ be a parabolic subsolution of \eqref{flow}. Then there exists a constant $\kappa=\kappa(\delta_0,R_0)$ such that if $|\lambda(\Hess\,u)|\geq R_0$ we have
			\begin{align*}
				&\mathrm{either}&&  L(u-\underline{u})\geq \kappa\mathcal{F}\,,\\
				&\mathrm{or }&& F^{k\bar{k}}(\Hess\,u)\geq \kappa\mathcal{F}\,, \qquad \text{for all } k=1,\dots,n\,.
			\end{align*}
			Here and hereafter we denote $\mathcal{F}=	\sum_{i=1}^{n}F^{i\bar{i}}(\Hess\,u)$.
		\end{proposition}
		\begin{proof}
			Using the argument of Phong-T\^o \cite[Lemma 3]{PT17}, we shall prove
			that $\underline{u}$ is also a parabolic $\mathcal{C}$-subsolution. As pointed in  \cite[Lemma 8]{PT17}, it suffices to show there exists a uniform constant $\epsilon>0$ such that for any $i=1,2,
			\cdots,n$,
			\begin{equation}\label{equivalent state}
				\lim_{s\rightarrow \infty} \cot\Theta(\underline{\lambda}+s\,\textbf{e}_{i})-\partial_{t}\underline{u}>\epsilon+\cot\hat{\theta},
			\end{equation}
			where we have used the simplified notation $\lambda_{j}(\Hess\,\underline{u})=\underline{\lambda}_{j}$. It is clear that $\lim_{s\rightarrow \infty} \Theta(\underline{\lambda}+s\textbf{e}_{i})=\sum_{j\neq i}\arccot\, \underline{\lambda}_{j}$. Therefore \eqref{equivalent state} is equivalent to 
			\begin{equation}\label{easier criterion of subsolution}
				\cot\Big(\sum_{j\neq i}\arccot\, \underline{\lambda}_{j}\Big)-(\partial_{t}\underline{u}+\cot\hat{\theta})>\epsilon,
			\end{equation}
			which can be verified if we set $\epsilon=\frac{1}{2}\min_{1\leq i\leq n}\inf_{\overline{\Omega}\times [0,\infty)}\arccot\, \underline{\lambda}_{i}.$
			
			Indeed, using the concavity of $\lambda\mapsto \cot\Theta(\lambda)$, we get 
			\[
			\begin{split}
				\arccot\, \underline{\lambda}_{i}=&\Theta(\underline{\lambda})-\sum_{j\neq i}\arccot\, \underline{\lambda}_{j}\\
				\leq & 	\cot\Big(\sum_{j\neq i}\arccot\, \underline{\lambda}_{j}\Big)-\cot \Theta(\underline{\lambda})\\
				\leq &\cot\Big(\sum_{j\neq i}\arccot\, \underline{\lambda}_{j}\Big)-(\partial_{t}\underline{u}+\cot\hat{\theta}).
			\end{split}
			\]
			Here, in the last inequality we have used \eqref{subsolution-1}. By the definition of $\epsilon$, then this proves \eqref{easier criterion of subsolution}.
		\end{proof}

		\subsection{The $\partial_{t}u$ estimate.}
		\begin{proposition}\label{hyperut} 
			On $\overline{\mathcal{Q}}_{T}$, there exists a constant $\delta_1\in (0,\frac{\pi}{4})$ depending on $\psi$, $\theta_0$ and $\hat{\theta}$ such that
			\begin{itemize}
				\item[(a)] $\frac{\pi}{2}-\delta_1> \Theta(\Hess\,u)>\delta_1.$
				\vspace{1mm}
				\item[(b)]  $|\partial_{t}u|\leq \frac{1}{\delta_{1}}.$
				\item[(c)]  $\mathcal{F}>\delta_1 .$
			\end{itemize}
		\end{proposition}
		
		\begin{proof}
			Differentiating  \eqref{flow} in $t$ variable we get $	L(\partial_{t}u)=0. $
			By maximum principle, we get  $\inf_{\partial_p \mathcal{Q}_T}\partial_{t}u\leq \partial_{t}u\leq \sup_{\partial_p \mathcal{Q}_T}\partial_{t}u$ and hence
			\[\inf_{\partial_p \mathcal{Q}_T}\Theta(\Hess\,\psi)\leq \Theta(\Hess\,u)\leq \sup_{\partial_p \mathcal{Q}_T}\Theta(\Hess\,\psi).\]
			Note that \[\sup_{\partial_p \mathcal{Q}_T}\Theta(\Hess\,\psi)=\max\{\sup_{\partial\Omega}\Theta(\Hess\,\varphi), \sup_{\mathcal{S}}\Theta(\Hess\,\psi)\}\leq  \frac{\pi}{2}-\theta_0\]
			by using \eqref{initial} and C2). Similarly, we  also get $\inf_{\partial_p \mathcal{Q}_T}\Theta(\Hess\,\psi)>0$. These complete the proof of (a) and (b).
			
			We remark that, as in \cite{PT17}, there exists a constant $\delta
			>0 $ depending on $\|\partial_{t}u\|_{C^0}$ such that $\mathcal{F}>\delta$. This proves (c). 
		\end{proof}
		\subsection{$C^{0}$ estimate}
		
		\begin{proposition}\label{C0}
			We have on $\overline{\mathcal{Q}}_{T}$, 
			\begin{equation}\label{c0}
				\underline{u}\leq u\leq \sup_{\mathcal{S}}\underline{u}.
			\end{equation}
			Therefore, there exists a uniform constant $C_{1}$ such that $|u|\leq C_{1}$ on $\overline{\mathcal{Q}}_{T}$.
		\end{proposition}
		\begin{proof}
			The first inequality in \eqref{c0} is follows from the comparison principle. If we order that $\lambda_{1}(\Hess\,u)\geq \cdots \geq \lambda_{n}(\Hess\,u)$  in $\mathcal{Q}_{T}$, then by Proposition \eqref{hyperut} we immediately have $\lambda_{n}(\Hess\,u)\geq \tan(\delta_1)$. This implies $u(\cdot,t)$ is plurisubharmonic in $\Omega$. For any $t\in [0,T]$, we see that $u(\cdot,t)$ achieves its maximum on $\partial\Omega$, i.e.
			\[u(\cdot,t)\leq \sup_{\partial\Omega}u(\cdot,t)=\sup_{\partial\Omega}\underline{u}(\cdot,t)\]
			for all $t\in [0,T]$, and then the second inequality follows.
		\end{proof}
		
		\subsection{Gradient estimate}
		Extend $\psi$ to a spatial harmonic function $\overline{u} $ which satisfies  $\Delta\overline{u}(\cdot,t)=0$ in $\Omega$ and $\overline{u}(\cdot,t)=\psi(\cdot,t)$ in $\overline{\Omega}$, for any $t\in [0,T]$. By maximum principle we conclude $\underline{u} \leq u\leq \overline{u}$ on $\overline{\mathcal{Q}}_{T}$  and $\underline{u}=u=\overline{u}$ on $\mathcal{S}$. 
		This also gives us
		\begin{equation}\label{boudary c1}
			|Du|\leq |D\underline{u}|+ |D \overline{u}|\leq C_{2}\quad \mathrm{on} \ \mathcal{S}
		\end{equation}
		for a uniform constant $C_{2}$.

		In the rest of this section, we shall give the interior gradient estimate. 
		\begin{proposition}
			There exists a uniform constant $C_{3}$ such that 
			\begin{equation*}\label{inter c1}
				|Du|\leq C_{3}\quad \mathrm{on} \ \overline{\mathcal{Q}}_{T}.
			\end{equation*}
		\end{proposition}
		\begin{proof}
			For any	$k=1,\cdots n$, let us consider the quantity 
			$$v= \pm u_{x_{k}}+|z|^{2}\quad \mathrm{on} \ \overline{\mathcal{Q}}_{T}.$$
			Suppose that $v$ attains its maximum at an interior point $p_{0}=(z_{0},t_{0})$ of $\overline{\mathcal{Q}}_{T}$, otherwise we are done.
			By an orthogonal transformation, we may further assume  the complex Hessian $(u_{i\bar{j}})$ is diagonal at $p_{0}$. It follows that $F^{i\bar j}$ is also diagonal at $p_{0}$ and 
			$$F^{i\bar{j}}=\frac{1+\cot^{2}
				\Theta(\Hess\,u)}{1+\lambda^{2}_{j}(\Hess\,u)}\delta_{ij}.$$
			It follows from Proposition \ref{hyperut} and the maximal principle that at $p_{0}$,
			\begin{equation*}
				\begin{split}
					Lv=&v_{t}-F^{i\bar{j}}v_{i\bar{j}}=\pm \big((\partial_{t}u)_{x_{k}}-F^{i\bar{j}}u_{i\bar{j}x_{k}}\big)-\mathcal{F}\leq -\delta_1<0,
				\end{split}
			\end{equation*}
			which is absurd. This implies $v$ must achieve its maximum on $\mathcal{S}$. Together with \eqref{boudary c1}, we obtain the global gradient estimate.
		\end{proof}

		\subsection{Interior second order estimate}
		\begin{proposition}\label{Interior second order estimate}
			There exists a uniform constant $C_{4}$ such that 
			\begin{equation}\label{inter c2}
				|D^{2}u|\leq C_{4}\Big(1+\max_{\mathcal{S}} |D^{2}u|\Big)\quad \mathrm{on} \ \overline{\mathcal{Q}}_{T}.
			\end{equation}
		\end{proposition}
		\begin{proof}
			Similar to the gradient estimate, let us consider a quantity 
			$$w= \Delta u+|z|^{2}\quad \mathrm{on} \ \overline{\mathcal{Q}}_{T}.$$

			\begin{claim}\label{maximun w}
				$w$ cannot attain its maximum at an interior point of $\overline{\mathcal{Q}}_{T}$.
			\end{claim}

			Given this, as well as $u(\cdot,t)$ is plurisubharmonic in $\overline{\Omega}$, the inequality \eqref{inter c2} follows.

			\begin{proof}[Proof of Claim \ref{maximun w}]
				We fix an arbitrary interior point $q_{0}=(z_{0},t_{0})$ of $\overline{\mathcal{Q}}_{T}$.
				By an orthogonal transformation, we may further assume  the complex Hessian $(u_{i\bar{j}})$ is diagonal at $q_{0}$. Using the concavity of $F$,
				\begin{equation*}
					\begin{split}
						Lw=&\partial_{t}w-F^{i\bar{j}}w_{i\bar{j}}= \Delta (\partial_{t}u)-F^{i\bar{j}}(\Delta u)_{i\bar{j}}-\mathcal{F}
						\leq -\delta_{1}<0,
					\end{split}
				\end{equation*}
				which is impossible. By arbitrariness of $q_{0}$, the claim follows.
			\end{proof}
			Consequently, we have finished the proof of Proposition \ref{Interior second order estimate}.		
		\end{proof}
		\subsection{Boundary second order estimates}
		In this subsection we shall prove 
		\begin{equation*}\label{boundary c2}
			\sup_{\mathcal{S}} |D^{2}u|\leq C_{5}
		\end{equation*}
		for a uniform constant $C_{5}.$ We need the following lemma.

		\begin{lemma}\label{key corollary}

			Let $\delta_{0}, R_0>0$ be uniform constants as in Definition \ref{parabolic sub}  and $\underline{u}$ be a subsolution of \eqref{dp}. There exists a constant $\delta_{3}=\delta_{3}(\delta_0,R_0,\psi,\underline{u},\theta_0, \hat{\theta})$ such that  if $|\lambda(\Hess\,u)|\geq R_0$, one of following statements holds: 
			\begin{align}\label{key inequality}
				&\mathrm{either}	\, & &L(u-\underline{u})\geq \delta_{3}\,,\\
				\label{good case}
				&\mathrm{or} \,  & & \frac{1}{\delta_{3}}\geq \lambda_{1}(\Hess\,u)\geq  \cdots \geq \lambda_{n}(\Hess\,u)\geq \tan (\delta_1)\, \  \mathrm{on}\ \overline{\mathcal{Q}}_T.
			\end{align}
			Here we have ordered $\lambda_{1}(\Hess\,u)\geq \cdots \geq \lambda_{n}(\Hess\,u)$.
		\end{lemma}
		\begin{proof}
			Using Lemma \ref{key}, there exists a constant $\kappa=\kappa(\delta_0,R_0)$ such that  one of the following inequalities holds:
			\begin{itemize}
				\item[(a)] $L(u-\underline{u})\geq \kappa\mathcal{F}$.
				\vspace{1mm}
				\item[(b)] $F^{k\bar{k}}(\Hess\,u)\geq \kappa\mathcal{F}$ for any $k=1,2,\cdots,n$.
			\end{itemize}
			
			If (a) holds, we immediately obtain \eqref{key inequality}. 
			Now we assume (b) holds. Then $F^{1\bar{1}}\geq \kappa F^{n\bar n}$ and hence $\lambda_n(\Hess\,u)\geq \sqrt{\kappa(\lambda^{2}_1(\Hess\,u)-1)}$. It follows from
			Propsition \ref{hyperut}(a)  that $\tan(\delta_1)\leq \lambda_n(\Hess\,u)\leq \cot(\frac{\delta_1}{n})$.
			These give  a uniform upper bound of $\lambda_1(\Hess\,u)$.
			
		\end{proof}
		
		\begin{remark}
			We remind that \eqref{good case} gives us the global second order estimates. In the sequel,  we  always assume \eqref{key inequality} holds. 
		\end{remark}

		Let us fix a point $p_{0}=(z,t_{0})\in \mathcal{S}$. Choose coordinates $z_1,z_2,\cdots, z_n$ with $z_i=x_i+\sqrt{-1}y_i$ for $\Omega$. After a translation we may assume $z=0$ and  the positive $x_{n}$ axis is the direction of interior normal of $\partial\Omega$ at 0. We denote $$s_{2k-1}=x_{k},\, s_{2k}=y_{k},\,
		1\leq k\leq n-1;\, s_{2n-1}=y_{n},\,s_{2n}=x_{n}$$ 
		and $s'=(s_{1},\cdots,s_{2n-1})$. Near the origin, we may 
		assume $\partial\Omega$ is represented as a graph
		\[x_{n}=\rho(s')=\frac{1}{2}\sum_{\alpha,\beta=1}^{2n-1}\rho_{\alpha\beta}(0)s_{\alpha}s_{\beta}+O(|s'|^{3}).\]
		\subsubsection{Pure tangential estimate}\label{pure}
		Since $(u-\underline{u})(s',\rho(s'),t)=0$, one derives
		\begin{equation}\label{osca}
			(u-\underline{u})_{s_{\alpha}s_{\beta}}(p_{0})=-(u-\underline{u})_{x_n}(p_{0})\rho_{\alpha\beta}(0), \qquad \alpha,\beta<2n.
		\end{equation}
		This gives the desired estimate.

		\subsubsection{Mixed direction estimate}  Next we proceed to prove 
		\begin{equation}\label{Mixed direction estimate}
			|u_{x_{n}s_{\alpha}}|(p_{0})\leq C_{5},\qquad \alpha<2n. 
		\end{equation}
		For a  small constant $\epsilon>0$ and set
		$ \mathcal{Q}_{\epsilon,p_{0}}=(\Omega\cap B_{\epsilon}(0))\times (0,t_{0}).$
		Let $d(z)=\text{dist}(z,\partial\Omega)$ be distance function from $\partial\Omega$.
		The following lemma plays a key role in our proof.
		\begin{lemma}\label{key lemma}
			There exist some uniform positive constants $s$ and $\epsilon$   small and $N$ large such that the function
			\begin{equation*}
				v=(u-\underline{u})+sd-Nd^{2}
			\end{equation*}
			satisfies
			\begin{equation}\label{Lv}
				\begin{cases}
					v\geq \frac{s}{2}d \quad &\mathrm{on} \, (B_{\epsilon}(0)\cap\Omega)\times \{0\}.\\[1mm]        v\geq 0 \quad &\mathrm{on} \, \partial(\Omega\cap B_{\delta}(0))\times (0,t_{0}].\\[1mm]
					L(v)\geq\epsilon_0 \quad &\mathrm{in} \, \mathcal{Q}_{\epsilon,p_{0}},
				\end{cases}
			\end{equation}
			where $\epsilon_0>0$ is a constant depending on $\varphi,\,\underline{u},\, \theta_0,\,\hat{\theta}$.
		\end{lemma}

		\begin{proof}
			Since $\underline{u}\leq u$ by \eqref{c0}, one derives $v\geq 0$ on $\partial(\Omega\cap B_{\epsilon}(0))\times (0,t_{0})$   if we further require $\epsilon\ll s\ll 1$ such that $N\epsilon<\frac{s}{2}$. In this case we also have the first inequality in \eqref{Lv}. Now we divide the proof of last inequality in \eqref{Lv} into two possibilities:
			\begin{itemize}
				\item[(a)] Assume $\sum_{i=1}^{n}u_{i\bar{i}}\leq R_0$ for the constant $R_0$  in Definition \ref{parabolic sub}. Since $|\partial d|=\frac{1}{2}$ and $F^{i\bar{i}}\geq \frac{1}{C(1+R_0^{2})}$,\begin{equation*}\label{ld}
					F^{i\bar{j}}d_{i}d_{\bar{j}}\geq \frac{1}{2C(1+R_0^{2})}, \qquad
					|F^{i\bar{j}}d_{i\bar{j}}|\leq C.
				\end{equation*}
				Moreover, by definition of $F^{i\bar j}$,
				$
				F^{i\bar{j}}(\underline{u}_{i\bar{j}}-u_{i\bar{j}})\leq C.
				$
				It follows from the above inequalities that
				\begin{equation*}
					\begin{split}
						L(v)&=	 \partial_{t}u-F^{i\bar{j}}(u_{i\bar{j}}-\underline{u}_{i\bar{j}})-(s-Nd)	F^{i\bar{j}}d_{i\bar{j}}+NF^{i\bar{j}}d_{i}d_{\bar{j}}\\
						&\geq  \frac{N}{2C(1+R_0^{2})}-C-(s-N\epsilon)C\geq\epsilon_0
					\end{split}
				\end{equation*}
				for a uniform constant $\epsilon_0>0$			when $N$ is large enough.
				\vspace{1mm}
				\item[(b)]
				Suppose now that $\sum_{i=1}^{n}u_{i\bar{i}}\geq R_0$, where $R_0$ is the constant as in Definition \ref{parabolic sub}. It follows from \eqref{key corollary} that
				\begin{equation*}
					\begin{split}
						L(u-\underline{u})\geq& \delta_{3}.
					\end{split}
				\end{equation*}
				In this case, we also have $|F^{i\bar{j}}d_{i\bar{j}}|\leq C$  and the term $F^{i\bar{j}}d_{i}d_{\bar{j}}$ is non-negative. By a direct calculation,
				\begin{equation*}
					\begin{split}
						L(sd-Nd^{2})=&-(s-Nd)F^{i\bar{j}}d_{i\bar{j}}+NF^{i\bar{j}}d_{i}d_{\bar{j}}
						\geq -C(s+N\epsilon).
					\end{split}
				\end{equation*}
				Thus, $Lv\geq \delta_{3}-C(s+N\epsilon)\geq\epsilon_0$ as $s,\epsilon$ are small enough.
			\end{itemize}
			
		\end{proof}

		For each integer $\alpha$ lies in $[1,2n-1]$, we define a tangential vector field
		\[T_{\alpha}=\partial_{s_{\alpha}}+\sum_{\beta=1}^{2n-1}\rho_{\alpha\beta}(0)\big(s_{\beta}\partial_{x_{n}}-x_{n}\partial_{s_{\beta}}\big).\]
		By a similar argument in \cite[p 89]{CPW17}, we conclude that
		\begin{equation}\label{similar argument}
			\begin{cases}
				\lim\sup_{(s',x_{n},t)\rightarrow (s',\rho(s'),t)}|T_{\alpha}(u-\underline{u})|&\leq C|s'|^{2}. \\[1mm]  
				|T_{\alpha}(u-\underline{u})|\leq C &\mathrm{on}\, (\partial B_\epsilon\cap \Omega)\times [0,t_0] .\\[1mm]
				|T_{\alpha}(u-\underline{u})|\leq Cd &\mathrm{on}\, (B_\epsilon\cap \Omega)\times \{0\} .\\[1mm]
				|LT_{\alpha}(u-\underline{u})|\leq C  &\mathrm{in} \,\mathcal{Q}_{\epsilon,p_{0}}. \\[1mm]
			\end{cases}
		\end{equation}
		Here the constant $C$ is universal.

		Let us employ the barrier function
		\[w=B_2v+B_1|z|^{2}\pm T_{\alpha}(u-\underline{u})\]
		with $B_1, B_2$ are positive constants.
		It follows from \eqref{Lv}-\eqref{similar argument} that when $B_2\gg B_1\geq C\epsilon^{-2}$,
		\begin{equation}\label{max principle for w}
			Lw> 0\quad \text{in} \, \mathcal{Q}_{\epsilon,p_{0}},\qquad
			\underset{x\rightarrow\partial_{p}\mathcal{Q}_{\epsilon,p_{0}}}{\lim\inf}w(x) \geq 0.
		\end{equation}
		Therefore, $w_{x_{n}}(p_{0})\leq 0$ and this gives the desired estimate.		
		\subsubsection{Pure normal estimate}		
		In this subsection, we shall prove 
		\begin{equation}\label{Pure normal estimate}
			|u_{x_{n}x_{n}}|\leq C_{5},\qquad \mathrm{on} \,\, \mathcal{S}.
		\end{equation}

		Before we present the proof of \eqref{Pure normal estimate}  it is necessary to use some well-known results from matrix theory.		\footnote{Hereafter, we let $\alpha,\beta=1,2,\cdots, n-1;$
			$i,j=1,2,\cdots,n.$	
		}For an Hermitian matrix $A=(a_{i\bar{j}})$  with eigenvalues $\lambda_{i}(A)$, we set $\tilde{A}=(a_{\alpha\bar{\beta}})$ and denote these eigenvalues of $\tilde{A}$  by $\lambda'_{\alpha}(\tilde{A})$. It follows from Cauchy-interlace inequality (see, e.g. \cite{Hwa04}) and \cite{CNS85} that
		\begin{itemize}
			\item $\lambda_{j}(A)\leq  \lambda'_{j}(\tilde{A}) \leq \lambda_{j+1}(A)$ for $j\leq n-1$.
			\vspace{0.1cm}
			\item $\lambda_{j}(A)=\lambda'_{j}(\tilde{A})+o(1)$ for $j\leq n-1$  as  $|a_{n\bar{n}}|\ri \infty$.
			\vspace{0.1cm}
			\item 
			$a_{n\bar{n}}\leq 	\lambda_{n}(A)\leq  a_{n\bar{n}}\Big(1+O\big(\frac{1}{a_{n\bar{n}}}\big)\Big)$ as  $|a_{n\bar{n}}|\ri \infty$.
		\end{itemize}

		Let $U=(u_{i\bar{j}})$ and $\underline{U}=(\underline{u}_{i\bar{j}})$ be Hessian matrices of $u$ and $\underline{u}$ respectively. It suffices to show that there are positive uniform constants $c_{0}, D_{0}$ such that for all $D\geq D_{0}$, we have
		$(\lambda'(\tilde{U}),D)\in \Gamma$ as well as
		\begin{equation*}\label{suffices}
			m_{D}=\underset{\mathcal{S}}{\min}\big\{\cot\Theta(\lambda'(\tilde{U}),D)-\cot\Theta(\lambda(U))\big\}\geq c_{0}.
		\end{equation*}
			
		To this end, we shall follow an idea of Trudinger \cite{Tru95}. 
		We remark that $m_D$ is nondecreasing with respect to $D$. Now we prove that
		\begin{equation}\label{reduced goal}
			\tilde{m}=\underset{D\rightarrow \infty}{\lim}m_{D}\geq c_{0}.
		\end{equation}

			In fact, if \eqref{hyperut} holds, there exist a constant $D_1$ such that 
	\begin{equation}
	\begin{split}
	\cot\theta\leq \cot\Theta(\lambda'(\tilde{U}),D_1)- \frac{c_0}{2}\\
		\end{split}
		\end{equation}
		If we set $\cot x= \cot\Theta(\lambda'(\tilde{U}),D_1)- \frac{c_0}{2}$, using mean value theorem, we have
		\[x\geq \Theta(\lambda'(\tilde{U}),D_1)+\frac{c_0}{2}\sin^2\Theta(\lambda'(\tilde{U}),D_1) ,\]
		which implies
			\[\Theta(\lambda(U)) \geq \Theta(\lambda'(\tilde{U}),D_1)+\frac{c_0}{2}\sin^2\Theta(\lambda'(\tilde{U}),D_1) .\]
			Therefore
			\begin{equation*}
			    \arccot\lambda_n\geq \arccot D_1+\arccot\Theta(\lambda'(\tilde{U}))-\sum_{i=1}^{n-1}\arccot\lambda_i+\frac{c_0}{2}\sin^2\Theta(\lambda'(\tilde{U}),D_1)
			\end{equation*}
			Here $\Theta(\lambda(U))$ means that $\Theta(\lambda(U))=\sum_{i=1}^{n-1}\arccot \lambda'_i$.
	Note that 	$\arccot D_1\leq \frac{1}{D_1}$	and $\arccot\Theta(\lambda'(\tilde{U}))-\sum_{i=1}^{n-1}\arccot\lambda_i=o(1)$. Then when $D_i$ is large enough, we have
	\begin{equation}
	   \frac{\pi}{2} \geq   \arccot\lambda_n\geq \frac{1}{C}.
	\end{equation}
	It follows $|u_{x_nx_n}|\leq C$.

		Suppose that $\tilde{m}<\infty$ (otherwise we are done) and  $\tilde{m}$ is achieved at a point $p_{0}=(0,t_{0})\in\mathcal{S}$. Choose a local coordinates $e_{1},\cdots,e_{n}$ with $\mathrm{Re}\,e_{n}=\partial_{x_{n}}$ is the 
		interior normal of $\partial\Omega$. 
		Let us denote
		$$\Gamma_{\infty}=\{(\lambda_{1},\cdots,\lambda_{n-1}) \mid \lambda_{\alpha}>0, \ 1\leq \alpha\leq n-1\}$$
		be a positive orthant in $\mathbb{R}^{n-1}$. We define 
		\[\widetilde{F}[E]=\lim_{D\rightarrow\infty}\cot\Theta(\lambda'(E),D),\qquad
		\widetilde{F}_0^{\alpha\bar{\beta}}=\frac{\partial 
			\widetilde{F}}{\partial E_{\alpha\bar{\beta}}}[\tilde{U}(p_0)]
		\]
		on the set of $(n-1)\times (n-1)$ Hermitian matrices with $\lambda'(E)\in \Gamma_{\infty}$. Note that $\widetilde{F}$ is concave since the operator $\lambda\mapsto\cot\Theta(\lambda)$ is concave.
		Hence,  
		\begin{equation}\label{concave}
			\widetilde{F}_0^{\alpha\bar{\beta}}\big(E_{\alpha\bar{\beta}}-\tilde{U}_{\alpha\bar{\beta}}(p_0)\big)\geq \widetilde{F}[E]-\widetilde{F}[\tilde{U}(p_0)].
		\end{equation}
		Similar to \eqref{osca}, one derives  $\tilde{U}_{\alpha\bar{\beta}}(p_0)-\tilde{\underline{U}}_{\alpha\bar{\beta}}(p_0)=-(u-\underline{u})_{x_{n}}(p_{0})\rho_{\alpha\bar{\beta}}(0).$
		Insert it into \eqref{concave} yields that
		\begin{equation}\label{concave1}
			\begin{split}
				(u-\underline{u})_{x_{n}}(p_{0})\widetilde{F}_0^{\alpha\bar{\beta}}\rho_{\alpha\bar{\beta}}(p_0) 
				\geq& \widetilde{F}[\tilde{\underline{U}}(p_0)]-\widetilde{F}[\tilde{U}(p_0)]	\\
				= &\widetilde{F}[\tilde{\underline{U}}(p_0)]-\tilde{m}-\cot\hat{\theta}-\partial_{t}u(p_0)\\
				= &\widetilde{F}(\tilde{\underline{U}}(p_0))-\tilde{m}-\cot\hat{\theta}-\partial_{t}\underline{u}(p_0)\\
				\geq &\widetilde{F}[\tilde{\underline{U}}(p_0)]-\cot\Theta(\lambda(\underline{U}(p_0)))-\tilde{m}\\
				\geq& \tilde{c}-\tilde{m}.
			\end{split}
		\end{equation}
		Here the constant $\tilde{c}=\lim_{D\rightarrow \infty}\min_{\mathcal{S}}[	\cot\Theta(\lambda'(\tilde{\underline{U}}),D)-\cot\Theta(\lambda(\underline{U}))]>0$ by the monotonicity of  $\lambda\mapsto\cot\Theta(\lambda)$.
		Now we divide the proof of \eqref{reduced goal} into two cases.

		\textbf{Case 1.} Assume
		$
		(u-\underline{u})_{x_{n}}(p_0)\widetilde{F}_0^{\alpha\bar{\beta}}\rho_{\alpha\bar{\beta}}(0) \leq \frac{\tilde{c}}{2}.
		$
		Then \eqref{concave1} gives us $\tilde{m}\geq \frac{\tilde{c}}{2}>0$ and we  are done.

		\textbf{Case 2.} Suppose now that
		\begin{equation}\label{case 2}
			(u-\underline{u})_{x_{n}}(p_0)\widetilde{F}_0^{\alpha\bar{\beta}}\rho_{\alpha\bar{\beta}}(0) \geq \frac{\tilde{c}}{2}.
		\end{equation}
		Let us denote $\eta=\widetilde{F}_0^{\alpha\bar{\beta}}\rho_{\alpha\bar{\beta}}.$
		Note that $(u-\underline{u})_{x_{n}}(p_{0})\geq 0$ and hence also strictly positive by \eqref{case 2}. It follows that
		\begin{equation}\label{lower bound of eta}
			\eta(p_{0})	\geq \frac{\tilde{c}}{2(u-\underline{u})_{x_{n}}(p_{0})}\geq 2\epsilon_1\tilde{c} 
		\end{equation}
		for some uniform constant $\epsilon_1>0$. We may assume $\eta\geq \epsilon_1\tilde{c}$ in a neighborhood $\mathcal{Q}_{\epsilon,p_0}$ of $p_0$ when $\epsilon$ is sufficiently small.

		Let us define a function $\Phi$ in $\mathcal{Q}_{\epsilon,p_0}$ by
		\begin{equation*}
			\begin{split}
				\Phi=&-\eta(u-\underline{u})_{x_{n}}+\widetilde{F}_0^{\alpha\bar{\beta}}	\big(\tilde{\underline{U}}_{\alpha\bar{\beta}}-\tilde{U}_{\alpha\bar{\beta}}(p_0)\big)		-\partial_{t}u+\partial_{t}u(p_0). \\
			\end{split}
		\end{equation*}
		A straightforward computation gives
		\begin{equation*}
			-\eta\,(u-\underline{u})_{x_{n}}=\widetilde{F}_0^{\alpha\bar{\beta}}	\big(\tilde{U}_{\alpha\bar{\beta}}-\tilde{\underline{U}}_{\alpha\bar{\beta}}\big).	
		\end{equation*}
		This together with \eqref{concave} yield that
		\begin{equation*}
			\begin{split}
				\Phi&=\widetilde{F}_0^{\alpha\bar{\beta}}	\big(\tilde{U}_{\alpha\bar{\beta}}-\tilde{U}_{\alpha\bar{\beta}}(p_0)\big)-\partial_{t}u+\partial_{t}u(p_0)\\
				&\geq  \widetilde{F}[\tilde{U}]- \widetilde{F}[\tilde{U}(p_0)]	-\partial_{t}u+\partial_{t}u(p_0)\\
				&= \widetilde{F}(\tilde{U})-\cot\hat{\theta}-\partial_{t}u- \tilde{m}.\\
			\end{split}
		\end{equation*}
		It follows that $\Phi(p_0)=0$ and $\Phi\geq 0$ on $\mathcal{S}\cap \overline{\mathcal{Q}}_{\epsilon,p_0}$. Let us consider the barrier function
		\begin{equation*}
			\Psi=\Phi+B_2 v+B_1|z|^{2} \qquad  \text{in}\, \mathcal{Q}_{\epsilon,p_0},
		\end{equation*}
		where $B_1,B_2$ are certain positive constants.	Using a similar argument as in \eqref{max principle for w}, we obtain \[\underset{x\rightarrow\partial_{p}\mathcal{Q}_{\epsilon,p_{0}}}{\lim\inf}\Psi\geq 0,\qquad L\Psi\geq 0 \quad \text{in}\ \mathcal{Q}_{\epsilon,p_0}\]
		as long as $B_2\gg B_1\geq C\epsilon^{2}$. Therefore, we have $\Phi_{x_n}(p_0)\geq -C_{5}$. This together with \eqref{lower bound of eta}  give a uniform upper bound for $u_{x_{n}x_{n}}(p_0)$.
		
		Now we are in a position of all the eigenvalues of $U$ are bounded from above, so $\lambda(U)(p_0)$  lies in a compact subset of $\Gamma$. Using the monotonicity of $\lambda\mapsto\cot\Theta(\lambda)$ again, 
		\begin{equation*}
			\tilde{m}\geq \tilde{m}_{D}=\cot\Theta(\lambda'(\tilde{U}
			(p_0)),D)-\cot\hat{\theta}-\partial_{t}u(p_0)>0
		\end{equation*}
		when $D$ is sufficiently large. 
		
		Consequently, we have finished the proof of \eqref{reduced goal} if we choose $c_0=\min\{\frac{\Tilde{c}}{2},\tilde{m}_D\}$ for a large constant $D$.
		
		\medskip
		Once we have tackled the global second order estimates, the equation \eqref{flow} is then uniformly parabolic. Using a standard PDE argument, one can also obtain the higher order estimates. We omit the standard step here.
		
		\section{Convergence}
		
		Consider the potential space
		\begin{equation*}
			\mathcal{H}=\{v\in C^{\infty}(\Omega)\ |\ \Theta(\Hess\, v)\in (0,\frac{\pi}{2}),\, v|_{\partial \Omega}=\vp\}.
		\end{equation*}

		\subsection{Calabi-Yau  functional}
		Fix $\phi\in\mathcal{H}$ and for each $\psi\in\mathcal{H}$, we define
		\begin{equation}\label{CY}
			\begin{split}
				CY_{\phi}(\psi)=&\int_{0}^1\int_{\Omega}\partial_{s}v \Big[(\Hess\, v+\sqrt{-1}\omega_0)^{n}-(\Hess\, \phi+\sqrt{-1}\omega_0)^{n}\Big]ds
				\\&+\int_{\Omega}\psi(\Hess\, \phi
				+\sqrt{-1}\omega_0)^{n},
			\end{split}
		\end{equation}
		where $v(\cdot,s)$ is a path in $\mathcal{H}$ which connects $\phi $ and $ \psi$, and $\omega_0=\sqrt{-1}\sum_{i=1}^{n}dz_{i}\wedge d\bar{z}_{i} $ is the standard metric in $\mathbb{C}^{n}$. 
		\begin{proposition}
			$CY_{\phi}(\psi)$ is well-defined and satisfies
			\begin{equation}\label{variation}
				\delta CY_{\phi}(\psi)=\int_{\Omega}(\delta\psi)(\Hess\, \psi +\sqrt{-1}\omega_0)^{n}.
			\end{equation}
		\end{proposition}
		\begin{proof}
			We  need to show that the integration defined in \eqref{CY} is independent of the choice of path $v(\cdot,s)$ in $\mathcal{H}$.

			Let $v(\cdot, s, r)$ be a family of paths in $\mathcal{H}$ satisfying $v(\cdot,0, r)=\phi$ and $v(\cdot, 1, r)=\psi$ with $ s, r\in[0,1]$.
			Let $ w(\cdot, \cdot,r)=\partial_{r}v(\cdot,\cdot, r)$ satisfying
			\[w(\cdot,\cdot,1)=w(\cdot,\cdot, 0)=0\text{\ \ and\ \ } w(z,\cdot,\cdot)=0 \text{\, for\, } z\in \partial \Omega.\]
			Then, using $w(\cdot,\cdot,1)=0$,
			\begin{equation*}
				\begin{split}
					\partial_r CY_{\phi}(\psi)=&\int_{0}^{1}\int_{\Omega}\partial_{s} w \Big[(\Hess\, v+\sqrt{-1}\omega_0)^{n}-(\Hess\, \phi+\sqrt{-1}\omega_0)^{n}\Big]ds\\
					&+\int_{0}^{1}\int_{\Omega}\partial_{s}v\cdot n\Hess\,w\wedge (\Hess\, v+\sqrt{-1}\omega_0)^{n-1}ds\\
					=&I_{1}+I_{2}.
				\end{split}
			\end{equation*}
			For the first term $I_1$, using integration by parts with respect to $s$ and $z$ respectively, we have
			\begin{equation*}
				\begin{split}
					I_{1}&=-\int_{0}^{1}\int_{\Omega} w\cdot n\Hess(\partial_{s}v)\wedge (\Hess\, u+\sqrt{-1}\omega_0)^{n-1}ds\\
					&=-\int_{0}^{1}\int_{\Omega}\partial_{s}v\cdot n\Hess\,w \wedge (\Hess\, u+\sqrt{-1}\omega_0)^{n-1}ds.
				\end{split}
			\end{equation*}
			Therefore $\partial_r CY_{\phi}(\psi)=0$. It follows that $CY_{\phi}(\psi)$ is well-defined. In a similar way one can also prove \eqref{variation}.
		\end{proof}
		\subsection{$J$-functional}		
		
		Let $u(\cdot,s)$ be a path in $\mathcal{H}$ from $u(\cdot,0)=\phi $ to $ u(\cdot,1)=\psi$.
		Now, we introduce the so-called $J$-functional  by \begin{equation}\label{J functional}
			J(u)=\mathrm{Im}(e^{-\sqrt{-1}\hat{\theta}}CY(u)).
		\end{equation}
		
		\begin{claim}\label{delta Ju}
			For each $u\in \mathcal{H}$,
			\begin{equation*}
				\begin{split}
					\mathrm{Im}\Big(e^{-\sqrt{-1}\hat{\theta}}&(\Hess\, u+\sqrt{-1}\omega_0)^{n}\Big)\\
					=&-\sin( \hat{\theta})\Big[\cot\Theta(\Hess\,u)-\cot\hat{\theta} \Big]\mathrm{Im}(\Hess\, u+\sqrt{-1}\omega_0)^{n}.
				\end{split}
			\end{equation*}
		\end{claim}

		\begin{proof}[Proof of Claim \ref{delta Ju}]
			By a direct calculation,
			\begin{equation*}
				\begin{split}
					\mathrm{Im}&\Big(e^{-\sqrt{-1}\hat{\theta}}(\Hess\, u+\sqrt{-1}\omega_0)^{n}\Big)\\
					=&\mathrm{Im}\Big((\cos(\hat{\theta})-\sqrt{-1}\sin( \hat{\theta}))\big(\mathrm{Re}(\Hess\, u+\sqrt{-1}\omega_0)^{n}+\sqrt{-1}\mathrm{Im}(\Hess\, u+\sqrt{-1}\omega_0)^{n}\big)\Big)\\
					=&\cos(\hat{\theta})\mathrm{Im}(\Hess\, u+\sqrt{-1}\omega_0)^{n}-\sin( \hat{\theta})\mathrm{Re}(\Hess\, u+\sqrt{-1}\omega_0)^{n}\\
					=&-\sin( \hat{\theta})\Big[\frac{\mathrm{Re}(\Hess\, u+\sqrt{-1}\omega_0)^{n}}{\mathrm{Im}(\Hess\, u+\sqrt{-1}\omega_0)^{n}}-\cot\hat{\theta} \Big]\mathrm{Im}(\Hess\, u+\sqrt{-1}\omega_0)^{n}\\
					=&-\sin( \hat{\theta})\Big[\cot\Theta(\Hess\,u)-\cot\hat{\theta} \Big]\mathrm{Im}(\Hess\, u+\sqrt{-1}\omega_0)^{n}.
				\end{split}
			\end{equation*}
			This completes the proof.
		\end{proof}

		By  Claim \ref{delta Ju}, we immediately see that the variation of $J$ is given by
		\begin{equation*}
			\delta J(u)=-\int_{\Omega}  (\delta u)\cdot \sin( \hat{\theta})\Big[\cot\Theta(\Hess\,u)-\cot\hat{\theta} \Big]\mathrm{Im}(\Hess\, u+\sqrt{-1}\omega_0)^{n}.
		\end{equation*}
		Then the  (negative) gradient flow of $J$-functional is exactly the heat flow \eqref{flow}, i.e.
		\begin{equation*}
			\frac{d}{dt} J(u)=-\int_{\Omega} (\partial_{t}u)^{2}\cdot \sin( \hat{\theta})\mathrm{Im}(\Hess\, u+\sqrt{-1}\omega_0)^{n}\leq 0.
		\end{equation*}
		Thus $J(u(\cdot,t))$ is monotonic decreasing as $t\rightarrow +\infty$.

		Combining with \eqref{J functional}, $J(u)$ is also bounded. Therefore, 
		\begin{equation}\label{integrablity}
			\int_{0}^{\infty}\int_{\Omega} (\partial_{t}u)^{2}\cdot\sin( \hat{\theta}) \mathrm{Im}(\Hess\, u+\sqrt{-1}\omega_0)^{n}dt<+\infty.
		\end{equation}
		We set 
		$S(t)=\int_{\Omega} (\partial_{t}u)^{2}\cdot\sin( \hat{\theta}) \mathrm{Im}(\Hess\, u+\sqrt{-1}\omega_0)^{n}\geq 0.$
		Thus \eqref{integrablity} implies 
		\begin{equation}\label{finite int}
			\int_{0}^{\infty}S(t)dt<\infty.
		\end{equation}
		
		\begin{claim}\label{St goes to 0}
			$S(t)\rightarrow 0$, as $t\rightarrow +\infty$.
		\end{claim}	
		
		\begin{proof}[Proof of Claim \ref{St goes to 0}]
			Differentiating \eqref{flow} in t, we have
			\begin{equation*}
				\begin{split}
					\partial^{2}_{t}u=&\frac{n\text{Re}(\Hess\,u\wedge (\Hess \,u+\sqrt{-1}\omega_0)^{n-1})}{\text{Im}(\Hess\,u+\sqrt{-1}\omega_0)^{n})}\\
					&-\cot \Theta(\Hess\,u) \frac{n\text{Im}(\Hess\,u\wedge (\Hess\,u+\sqrt{-1}\omega_0)^{n-1})}{\text{Im}(\Hess\,u+\sqrt{-1}\omega_0)^n)}.    
				\end{split}
			\end{equation*}
			Based on the a priori estimates in Section 2, it follows that
			\begin{equation*}
				\begin{split}
					S'=\int_{\Omega}&2n \partial_{t}u\sin (\hat{\theta})\mathrm{Re}(\Hess\,u\wedge (\Hess\,u+\sqrt{-1}\omega_0)^{n-1})\\
					&-\int_{\Omega}2n\partial_{t}u\cos(\hat{\theta})\mathrm{Im}(\Hess\,u\wedge (\Hess\,u+\sqrt{-1}\omega_0)^{n-1})\\
					&+\int_{\Omega}n (\partial_{t}u)^2\sin (\hat{\theta})\mathrm{Im}(\Hess\,u\wedge (\Hess\,u+\sqrt{-1}\omega_0)^{n-1})\\
					\leq C\int_{\Omega}&\big((\partial_{t}u)^2+|\partial_{t}u|\big)\omega_0^{n}
					\leq  C (S+\sqrt{S}),
				\end{split}
			\end{equation*}
			where in the last inequality we have used  Proposition \ref{hyperut} to derive $\mathrm{Im}(\Hess\,u+\sqrt{-1}\omega_0)^n\geq \frac{1}{C}\omega_{0}^{n}$. By solving the above ODE problem, we get
			\begin{equation}\label{self-control}
				\sqrt{S(t)}+1\leq e^{C(t-s)}(\sqrt{S(s)}+1),\qquad \forall\, t>s.
			\end{equation}
			Using \eqref{finite int}, we have for any positive constant $\epsilon$, there exists $t_{0}=t_{0}(\epsilon)$ such that  
			$
			|\{t>t_0: \, S(t)>\epsilon\}|\leq \epsilon.
			$
			Therefore,
			for any $t\geq t_0$, we can find a $t'\in \{t>t_0: \, S(t)\leq \epsilon\}$ with $0<t-t'\leq \epsilon$. Thanks to \eqref{self-control}, it follows that
			\begin{equation*}
				\sqrt{S(t)}\leq \sqrt{S(t')}e^{C\epsilon}+e^{C\epsilon}-1
				\leq \sqrt{\epsilon} e^{C\epsilon}+C\epsilon e^{C\epsilon}.
			\end{equation*}
			By arbitrariness of $\epsilon$, the Claim \ref{St goes to 0} follows.
		\end{proof}	
		
		\subsection{Proof of Theorem \ref{main theorem}}
		Since $\Omega$ is bounded, it follows from \eqref{St goes to 0} that
		\begin{equation*}
			\lim_{t\rightarrow \infty}\|\partial_{t}u(\cdot, t)\|_{L^{1}(\Omega)}=0.
		\end{equation*}
		By mean value theorem,
		\[|\cot\Theta(\Hess\,u)(\cdot, t)-\cot \hat{\theta}|\geq |\Theta(\Hess\,u)(\cdot,t)-\hat{\theta}|,\]
		which implies $\Theta(\Hess\,u)(\cdot, t)$ converges to $\hat{\theta}$ in $L^{1}(\Omega)$ as $t$ goes to $\infty$.

		Let us complete the proof of Theorem  \ref{main theorem}.
		Since $u(\cdot, t)$ are uniform $C^{\infty}$ bounded, using Arzela-Ascoli Theorem, there exists a subsequence $u(\cdot, t_{i})$ which converges  to   $u_{\infty}$ in $C^{\infty}$-topology.
		
		Now we prove $u(\cdot,t)$ converges to a function  $u_{\infty}$ in $\overline{\Omega}$ in $C^{\infty}$-topology. Indeed, suppose not. Then there exists an integer $k$, a constant $\epsilon>0$ and a time sequence  $t'_i\rightarrow \infty$ such that
		\begin{equation}\label{contradiction}
			\|u(\cdot, t'_i)-u_{\infty}\|_{C^{k}}\geq \epsilon, \qquad \text{\ for all \ } i.
		\end{equation}
		Since $u(\cdot, t'_i)$ are uniform $C^{k}$ bounded, using Arzela-Ascoli Theorem again, there exists a subsequence $u(\cdot, t'_{i_j})$ which converges  to   $u'_{\infty}$ in $C^{k}$-topology. Then by \eqref{contradiction}, we have
		\begin{equation*}\label{neq}
			\|u'_{\infty}-u_\infty\|_{C^{k}}\geq \epsilon,
		\end{equation*}
		which implies $u'_{\infty}\neq u_{\infty}$. 
		Letting $t_{i}\rightarrow \infty$ and $t'_{i_j}\rightarrow \infty$ in the \eqref{flow}, we know that $u_{\infty}$ and $u'_{\infty}$ are two distinct solutions of \eqref{dp}.   This is impossible by maximum principle.
		\qed


\begin{thebibliography}{99}
			
			\bibitem{CNS85} L. Caffarelli,  L. Nirenberg, and J. Spruck, The Dirichlet problem for nonlinear second-order elliptic equations. III. Functions of the eigenvalues of the Hessian. Acta Math. 155 (1985), no. 3--4, 261--301.
			
			\bibitem{CKNS85} L. Caffarelli, J.-J. Kohn, L. Nirenberg, and J. Spruck, The Dirichlet problem for nonlinear second-order elliptic equations II. Complex \MA, and uniformly elliptic equations. Comm. Pure Appl. Math., 38(2) (1985), 209--252.
			
			
			\bibitem{CL} J. Chu, M.-C. Lee, Hypercritical deformed Hermitian-Yang-Mills equation, preprint, arXiv: 2206.00387.
			
			\bibitem{CJY}{T. Collins, A. Jacob} and  {S.-T. Yau,} { (1,1) forms with specified Lagrangian phase: a priori estimates and algebraic obstructions,} Camb. J. Math. \textbf{8} (2020), no. 2, 407--452.
			
			
			\bibitem{CPW17} T.-C. Collins, S. Picard and X. Wu, Concavity of the Lagrangian phase operator and applications. Calc. Var. (2017) 56--89.
			
			\bibitem{CY}    {T.-C Collins},  {S.-T. Yau,} {Moment maps, nonlinear PDE, and stability in mirror symmetry}, I: geodesics. Ann. PDE 7 (2021), no. 1, Paper No. 11, 73 pp.
			
			\bibitem{FYZ20} J. Fu, S.-T. Yau, D. Zhang, A deformed Hermitian Yang-Mills flow, preprint, arXiv:2105.13576.
			
			\bibitem{Gu98} B. Guan,  The \DP~ for complex \MA~ equations and regularity of the pluri-complex Green function. Communications in analysis and geometry. Volume 6, Number 4,687-703,1998.
			
			\bibitem{Guan14}B. Guan, { Second-order estimates and regularity for fully nonlinear elliptic equations on Riemannian manifolds}. Duke Math. J. {\bf 163} (2014), no. 8, 1491--1524.
			
			\bibitem{HouLi} Z. Hou, Q. Li, Energy functionals and complex Monge-Amp\`ere equations. J. Inst. Math. Jussieu 9 (2010), no. 3, 463--476.
			
			\bibitem{HL82} F.-R. Harvey,  H.-B. Lawson, Calibrated geometries, Acta Math. {\bf 148} (1982), 47--157.
			
			\bibitem{HL09} F.-R. Harvey, H.-B. Lawson,  Dirichlet duality and the nonlinear Dirichlet problem, Comm. Pure Appl. Math. {\bf 62} (2009), 396--443.
			
			\bibitem{Hwa04} S. Hwang, Cauchy's interlace theorem for eigenvalues of Hermitian matrices. Am. Math. Mon. 111(2), 157--159 (2004).
			
			\bibitem{JY} {A. Jacob}, {S.-T.  Yau,} {A special Lagrangian type equation for holomorphic line bundles,} Math. Ann. \textbf{369} (2017), no. 1-2, 869--898.
			
			\bibitem{PS} {D.-H. Phong, J. Sturm,} The Monge-Amp\`ere operator and geodesics in the space of K\"ahler potentials. Invent. Math. 166 (2006), no. 1, 125--149.
			
			\bibitem{PT17} {D.-H. Phong, D.T. T\^{o}},  Fully nonlinear parabolic equations on compact Hermitian manifolds. Ann. Sci. \'Ec. Norm. Sup\'er. (4) 54 (2021), no. 3, 793--829.
			
			\bibitem{SG} {G. Sz\'{e}kelyhidi,} {Fully nonlinear elliptic equations on compact Hermitian manifolds,}  J. Differential Geom. \textbf{109} (2018), no. 2, 337--378.
			
			
			
			\bibitem{TR} {R. Takahashi,} {Tan-concavity property for Lagrangian phase operators and applications to the tangent Lagrangian phase flow,}  Internat. J. Math. 31 (2020), no. 14, 2050116, 26 pp.
			
			\bibitem{Tru95} N.-S. Trudinger,  On the Dirichlet problem for Hessian equations. Acta Math. 175 (1995), no. 2, 151--164.
			
			
			
			
			
			
		\end{thebibliography}
\end{document}